\documentclass[11pt]{amsart}
\usepackage{graphicx}
\newtheorem{thm}{Theorem}[section]

\newtheorem{lem}[thm]{Lemma}

\theoremstyle{definition}

\theoremstyle{remark}
\newtheorem{rem}[thm]{Remark}
\theoremstyle{definition}

\begin{document}

\title[Solution to the Erd\H{o}s problem]{Solution to the Erd\H{o}s problem on distinct residues of factorials}%
\author{Vyacheslav M. Abramov}%
\address{24 Sagan Drive, Cranbourne North, Victoria 3977, Australia}%

\email{vabramov126@gmail.com}%

\thanks{Former affiliation: School of Mathematics, Monash University, Australia}%
\thanks{Current status: aged pensioner}
\thanks{Address: 24 Sagan Drive, Cranbourne North, Melbourne, Victoria-3977, Australia}
\thanks{Email: vabramov126@gmail.com}
\thanks{ORCID: 0000-0002-9859-100X}

\begin{abstract}
Paul Erd\H{o}s posed the following question: \textit{Is there a prime number $p>5$ such that the residues of $2!$, $3!$,\ldots, $(p-1)!$ modulo $p$ all are distinct.} In this study, we give the negative answer on this question in a relatively simple way.
\end{abstract}

\maketitle

\section{Introduction}
Paul Erd\H{o}s posed the following question: \textit{Is there a prime number $p>5$ such that the residues of $2!$, $3!$,\ldots, $(p-1)!$ modulo $p$ all are distinct.} Trudgian \cite{T} called prime numbers satisfying this property \textit{socialist primes}. For the primes having the form $p=3 \ (\mathrm{mod}\ 4)$, the answer on this question is negative, due to Wilson's theorem and elementary modular arithmetic. Recall that
 Wilson's theorem states that a natural number $n > 1$ is a prime number if and only if $(n-1)!\equiv-1(\mathrm{mod}\ n)$, see \cite{W1}. According to an elementary modular arithmetic, along with $(p-1)!\equiv-1(\mathrm{mod}\ p)$ we also have $(p-2)!\equiv1(\mathrm{mod}\ p)$, and in the case of $p=3 \ (\mathrm{mod}\ 4)$, $\left[\left(\frac{p-1}{2}\right)!\right]^2\equiv1(\mathrm{mod}\ p)$ (see relation \eqref{1} below or \cite[Theorem 114]{HW}). Thus the question of Paul Erd\H{o}s generally reduces to the case of primes having the presentation $p=1 \ (\mathrm{mod}\ 4)$.

The problem of the existence of socialist prime numbers remains unsolved and appears in the list of unsolved problems in number theory in \cite[Sect. F11]{G}. Several studies have been conducted this problem. In all these studies the authors found the conditions on prime numbers, under which the answer can be positive. The complexity of those conditions made hard to answer the question whether such primes exist.

The first study was conducted by Rokowska and Schnitzel \cite{RS} in 1960.
They proved that if $p$ exists, then it should satisfy the following conditions
$
\quad p=5 (\mathrm{mod}\ 8),\quad \left(\frac{5}{p}\right)=-1, \quad \left(\frac{-23}{p}\right)=1.
$
Based on the aforementioned conditions, they showed numerically that there are no prime numbers satisfying the aforementioned property in the interval  between $5$ and $1000$.
In addition to that Rokowska and Schnitzel \cite{RS} also proved that if a socialist prime number $p$ exists, then for that $p$ none of the numbers $2!$, $3!$,\ldots, $(p-1)!$ is congruent modulo $p$ to $-\left(\frac{p-1}{2}\right)!$.

Extending the study of \cite{RS}, Trudgian \cite{T} assumed additionally that either $\left(\frac{1957}{p}\right)=1$, or $\left(\frac{1957}{p}\right)=-1$ with $\left(\frac{4y+25}{p}\right)=-1$ for all $y$ satisfying the equation $y(y+4)(y+6)-1\equiv0 \ (\mathrm{mod}\ p).$
By numerical calculations, he confirmed that socialist prime numbers less than $10^9$ do not exist.

Andreji\'{c} and Tatarevi\'{c}  \cite{AT2, AT1} also studied the problem of socialist prime numbers, where they established a connection with left factorial function as well as provided intensive numerical studies of residues confirming that socialist primes less than $10^{11}$ do not exist. Later Andreji\'{c}, Bostan and Tatarevi\'{c} \cite{ABT} used an improved algorithm to extend their numerical results confirming absence of socialist primes for the numbers smaller than $2^{40}$.

All of the aforementioned studies \cite{AT2, RS, T} were based on the direct analysis of residues, and the study that started in \cite{RS} was then developed in \cite{T} and \cite{AT2} following the same basic idea to find the conditions on prime numbers, under which the required prime number may exist.

The proof presented in this study is based on another idea. Unlike the previous studies, our study is not based on analysis of residues for the undecided cases. We prove that the known identities that follow from Wilson's theorem and modular arithmetic (see identities \eqref{1} below) contradict to the claim that for any $p>5$ the residues modulo $p$ of $\{2!, 3!,\ldots, (p-1)!\}$ are distinct. First we formulate and prove a closely related problem that then adapted to the problem formulated by Erd\H{o}s.
Our proof is fully based on combinatorial arguments.
\smallskip

Our main result is the following theorem.

\begin{thm}\label{thm1}
There are no socialist prime numbers.
\end{thm}

This theorem is proved in Section \ref{S2} that is organized as follow. In Section \ref{S2.1}, we present key identities \eqref{1} and on their basis reformulate the problem. In Section \ref{S2.2}, we solve an auxiliary problem that is closely related to the original one. The key result of this section that solves the aforementioned auxiliary problem is Lemma \ref{prop1}. In the proof of this lemma we introduce the concept of the perfect system of identities that is central in this study as well as provide the classification of the identities.
The classification of the identities is further applied in Section \ref{S2.4},  where \eqref{1} is represented by \eqref{24}--\eqref{23}. Then the system of identities \eqref{24} is further transformed to the form that enables us to finalize the proof of Theorem \ref{thm1}.

\section{Proof of Theorem \ref{thm1}}\label{S2}
\subsection{Key identities and reformulation of the problem}\label{S2.1}
Recall the known identities based on Wilson's theorem and modular arithmetic:
\begin{equation}\label{1}
\begin{aligned}
(p-1)!&=-(p-2)!=2!(p-3)!=-3!(p-4)!=...\\
&=(-1)^{i-1}(i-1)!(p-i)!=\ldots\\
&=(-1)^{\frac{p-1}{2}}\left(\frac{p-1}{2}\right)!\left(\frac{p-1}{2}\right)!\equiv{-1}\ (\mathrm{mod}\ p).
\end{aligned}
\end{equation}

Our aim is to prove that if the residues modulo $p$ of $\{2!, 3!,\ldots, (p-1)!\}$ are distinct, then \eqref{1} is impossible and vice versa, that is, given \eqref{1} the residues modulo $p$ of $\{2!, 3!,\ldots, (p-1)!\}$ cannot be distinct.

\smallskip

To prove this, we solve the auxiliary problem formulated in the next section.
\subsection{Auxiliary problem}\label{S2.2}

Let $\big\{\alpha_1, \alpha_2,\ldots, \alpha_{\frac{p-1}{2}}\big\}$ and let $\big\{\beta_1, \beta_2,\ldots, \beta_{\frac{p-1}{2}}\big\}$ be the two subsets of $I=\{1,2,\ldots,p-1\}$ satisfying the properties
\begin{eqnarray*}
  \big\{\alpha_1, \alpha_2,\ldots, \alpha_{\frac{p-1}{2}}\big\}\cup \big\{\beta_1, \beta_2,\ldots, \beta_{\frac{p-1}{2}}\big\}&=& I, \\
   \big\{\alpha_1, \alpha_2,\ldots, \alpha_{\frac{p-1}{2}}\big\}\cap \big\{\beta_1, \beta_2,\ldots, \beta_{\frac{p-1}{2}}\big\} &=& \emptyset,
\end{eqnarray*}
and specified as
\begin{equation}\label{2}
\alpha_i\beta_i\equiv(-1)^i\ (\mathrm{mod}\ p), \quad i=1, 2,\ldots, \frac{p-1}{2}.
\end{equation}
Our first aim is to find the values $p$ under which \eqref{2} is consistent, and describe the structure of systems of identities that provides this consistency.

\medskip

Our study is conducted under the assumption $p=1\ (\mathrm{mod}\ 4)$.
The obtained results will be then applied to the residues modulo $p$ of $\{2!$, $3!,\ldots$, $(p-1)!\}$ that according to \eqref{1} behave similarly.

\smallskip
We prove the following result.

\begin{lem}\label{prop1}
Assume that $p=1\ (\mathrm{mod}\ 4)$. Then, under an appropriate choice of $\alpha_i$ and $\beta_i$ in \eqref{2}, the system of identities is consistent if and only if $p=5\ (\mathrm{mod}\ 8)$.
\end{lem}

\begin{proof}
The proof is structured into four sections. In the first section, we describe the structure of the system of identities \eqref{2}. In the second section and third sections, we study the cases $p=1\ (\mathrm{mod}\ 8)$ and $p=5\ (\mathrm{mod}\ 8)$, respectively.
\smallskip

\textit{1. Basic structure of the system of identities.} Notice first that if $\alpha_1=1$, then $\beta_1$ must be equal to $p-1$ and the right-hand side must be $-1 \ (\mathrm{mod}\ p)$. Therefore if we split the system of identities \eqref{2} into two separate subsystems:
\begin{equation}\label{3}
\alpha_{2i}\beta_{2i}\equiv1 \ (\mathrm{mod}\ p),\quad i=1, 2,\ldots, \frac{p-1}{4},
\end{equation}
and
\begin{equation}\label{4}
\alpha_{2i-1}\beta_{2i-1}\equiv-1 \ (\mathrm{mod}\ p), \quad i=1, 2,\ldots, \frac{p-1}{4},
\end{equation}
then the aforementioned identity belongs to \eqref{4}. If we exclude this identity from consideration, we will have $\frac{p-1}{2}-1$ remaining identities in total, $\frac{p-1}{4}$ of which satisfy \eqref{3}, and the rest $\frac{p-1}{4}-1$ satisfy \eqref{4}. We will show below that one of $\frac{p-1}{4}$ identities of \eqref{3} has a specific form and can be excluded from consideration as well. We will call it \textit{tag} identity. That is, \eqref{3} consists of a tag identity and $\frac{p-1}{4}-1$ other (regular) identities. A tag identity, marked as
\[
\alpha_{\frac{p-1}{2}}\beta_{\frac{p-1}{2}} \equiv 1\ (\mathrm{mod}\ p),
\]
and having the form
\begin{equation}\label{31}
\alpha_{\frac{p-1}{2}}\left(p-\alpha_{\frac{p-1}{2}}\right) \equiv 1\ (\mathrm{mod}\ p),
\end{equation}
i.e., $\beta_{\frac{p-1}{2}}=p-\alpha_{\frac{p-1}{2}}$,
will be discussed later. It will be shown that there is only a unique identity in the system having the form $i(p-i)\equiv1 \ (\mathrm{mod}\ p)$. That is, a tag identity is \textit{the} tag identity.

Let us now discuss the properties of the remaining systems of regular identities. At this point, we assume that the system of identities \eqref{12} given below contains the $\frac{p-5}{4}$ identities, none of which satisfying the aforementioned property $i(p-i)\equiv1 \ (\mathrm{mod}\ p)$. That is, in any of the identities in \eqref{12}, $\beta_{2i}\neq p-\alpha_{2i}$.

We have:
\begin{equation}\label{12}
\alpha_{2i}\beta_{2i}\equiv1 \ (\mathrm{mod}\ p), \quad i=1, 2,\ldots, \frac{p-5}{4},
\end{equation}
and
\begin{equation}\label{13}
\alpha_{2i+1}\beta_{2i+1}\equiv-1 \ (\mathrm{mod}\ p), \quad i=1, 2,\ldots, \frac{p-5}{4}.
\end{equation}

Notice that if for some $i_0$,
\begin{equation}\label{14}
\alpha_{i_0}\beta_{i_0}\equiv1 \ (\mathrm{mod}\ p),
\end{equation}
then also
\begin{equation}\label{15}
(p-\alpha_{i_0})(p-\beta_{i_0})\equiv1 \ (\mathrm{mod}\ p).
\end{equation}
Similarly, if
\begin{equation}\label{16}
\alpha_{i_1}\beta_{i_1}\equiv-1 \ (\mathrm{mod}\ p),
\end{equation}
then also
\begin{equation}\label{17}
(p-\alpha_{i_1})(p-\beta_{i_1})\equiv-1 \ (\mathrm{mod}\ p).
\end{equation}
According to the assumption, all $\alpha_i$ and $\beta_i$ that appear in \eqref{12} and \eqref{13} must be distinct. Therefore after changing the variables, $\tilde{\alpha}_{i_0}=p-\alpha_{i_0}$, $\tilde{\beta}_{i_0}=p-\beta_{i_0}$, $\tilde{\alpha}_{i_1}=p-\alpha_{i_1}$ and $\tilde{\beta}_{i_1}=p-\beta_{i_1}$ are new distinct values.

Let us now discuss \eqref{31}.
It cannot be a part of the system of identities \eqref{14} and \eqref{15}, since if it satisfies \eqref{14} it is also satisfies \eqref{15}.

\smallskip
Now we prove that it is a unique identity in the system of identities \eqref{3}.

\smallskip
To specificate the tag identity, we consider the equation
\begin{equation}\label{20}
i(p-i)=Kp+1
\end{equation}
for $K$ and $i$. This equation reduces to
\[
K=i-\frac{i^2+1}{p},\quad 2\leq i\leq p-2.
\]
Apparently that $K$ takes integer value only in the case when $i^2\equiv-1\ (\mathrm{mod}\ p)$. Such value $i$ exists and is unique, if we dismiss another value $(p-i)^2\equiv-1\ (\mathrm{mod}\ p)$ satisfying the same properties due to the symmetry. That is, there is exactly one identity satisfying \eqref{20}, and the tag identity is strictly unique in the system of identities \eqref{2} or \eqref{3}.

For the further studies of \eqref{12} and \eqref{13}, we consider the two cases given below.

\smallskip
\textit{2. Case of $p=1\ (\mathrm{mod}\ 8)$.}
In this case, $\frac{p-1}{4}$ is even. Then $\frac{p-1}{4}-1$ is odd, and each of the systems of identities \eqref{12}, \eqref{13} contains the odd number of identities. This means that there exists $i_0$ such that the system of identities \eqref{12} contains \eqref{14} and does not contain \eqref{15}.

Then instead of \eqref{15} we must have
\begin{equation}\label{18}
\alpha_{i_0}(p-\beta_{i_0}) \equiv -1 \ (\mathrm{mod}\ p),
\end{equation}
or
\begin{equation}\label{18.1}
(p-\alpha_{i_0})\beta_{i_0} \equiv -1 \ (\mathrm{mod}\ p).
\end{equation}

However both $p-\beta_{i_0}$ and $p-\alpha_{i_0}$ appear in \eqref{15}, and therefore in both \eqref{12} and \eqref{13}.

We arrived at the contradiction that shows that the system of identities cannot be consistent.

\smallskip

\textit{3. Case of $p=5\ (\mathrm{mod}\ 8)$.}
In this case, $\frac{p-1}{4}$ is odd. Then $\frac{p-1}{4}-1$ is even, and each of the systems of identities \eqref{12}, \eqref{13} contains the even number of identities. This case is more complex.
To resolve this case, we introduce a notion of \emph{perfect system} that is a basic notion in this study.

The system of identities
\begin{eqnarray}
\alpha_2\gamma_2&\equiv&1\ (\mathrm{mod}\ p),\nonumber\\
\alpha_3\gamma_3&\equiv&1\ (\mathrm{mod}\ p),\nonumber\\
\ldots         &\equiv&\ldots,\label{21}\\
\alpha_{\frac{p-3}{2}}\gamma_{\frac{p-3}{2}}&\equiv&1\ (\mathrm{mod}\ p)\nonumber
\end{eqnarray}
is called perfect system. Apparently this system exists and unique. For any $\alpha_i$,
\[
\gamma_i:=\frac{K_ip+1}{\alpha_i},
\]
where $K_i$ is the smallest positive integer given such that the left-hand side of the equality is integer, $1\leq K_i\leq\alpha_i-1$. A simple example of the perfect system is given in Table \ref{tab:tab1} for $p=13$.

\begin{table}[h!]
\begin{center}
\caption{Example of perfect system for $p=13$}
\label{tab:tab1}
\begin{tabular}{c|c|c}
\hline\hline
$i$ &$\alpha_i$ &$\gamma_i$ \\
\hline\hline
$2$ & $2$ & $7$ \\
$3$ & $10$ & $4$\\
$4$ & $6$ & $11$ \\
$5$ & $3$ & $9$ \\
\hline
\end{tabular}
\end{center}
\end{table}

The system of identities \eqref{21} contains $2\left(\frac{p-1}{4}-1\right)$ identities, that is, all the identities except the first one, where $\alpha_1=1$ (or $\alpha_1=p-1$)
and the tag identity. The numeration of the indices coincides with that given in the joint system of identities \eqref{12} and \eqref{13}. The system of identities \eqref{21} is uniquely defined (but not unique; see Remark \ref{rem2} below), all $\alpha_i$ and $\gamma_i$ are distinct. Indeed, with $\alpha_{i_0}\neq\alpha_{i_1}$, we obviously have $\gamma_{i_0}\neq\gamma_{i_1}$. If we assume to the contrary that $\gamma_{i_0}=\gamma_{i_1}$, then we easily arrive at the contradiction by subtracting the identities. As well, $\alpha_i$, $i\geq3$, is chosen such that $\alpha_i$ distinguishes from all of the previous values $\alpha_j$ and $\gamma_j$, $2\leq j\leq i-1$.

The system is called \textit{perfect}, because all the $\alpha_i$ and $\gamma_i$ in this system are distinct and the system itself is uniquely defined. Note that the similar perfect system can be built for the identities, in which the right-hand side is equal to $-1 \ (\mathrm{mod}\ p)$.

Now, in order to get \eqref{12} and \eqref{13} from \eqref{21}, we separate half of the identities from \eqref{21} and multiply both sides there by $-1$ given that each half (subsystem) is constructed according to the rule: if \eqref{14} belongs to the subsystem, then also \eqref{15} does. Similarly, if \eqref{16} belongs to the other subsystem, then also \eqref{17} does. With closed subsystems, \eqref{12} and \eqref{13} are fully separated and the system of identities is consistent.

More specific explanation is as follows.
We take the first identity in \eqref{21} and transform it as follows
\[
\alpha_2(p-\gamma_2)=\alpha_2\tilde{\gamma_2}\equiv-1\ (\mathrm{mod}\ p).
\]
Then,
\[
(p-\alpha_2)\gamma_2=\tilde{\alpha_2}\gamma_2\equiv-1\ (\mathrm{mod}\ p).
\]
Thus instead of the original two identities
\begin{eqnarray*}
  \alpha_2\gamma_2 &\equiv& 1\ (\mathrm{mod}\ p), \\
  \tilde{\alpha}_2\tilde{\gamma}_2 &\equiv& 1\ (\mathrm{mod}\ p),
\end{eqnarray*}
we get
\begin{eqnarray*}
  \alpha_2\tilde{\gamma_2} &\equiv& -1\ (\mathrm{mod}\ p), \\
  \tilde{\alpha_2}\gamma_2 &\equiv& -1\ (\mathrm{mod}\ p).
\end{eqnarray*}
That is, in both of the cases the same set of four parameters $\alpha_2$, $\gamma_2$, $\tilde{\alpha}_2$ and $\tilde{\gamma}_2$ is used. This procedure continues similarly with other quantities until getting a half part of all quantities transformed and collected in the subsystem.
\end{proof}

\begin{rem}\label{rem1}
According to the construction in the proof of Lemma \ref{prop1}, the system of identities \eqref{2} to be consistent, must be originated from the perfect system of identities, according to the rules established in the proof. That is, we originally must have \eqref{21} and then build \eqref{12} and \eqref{13} with the following extension to \eqref{2}. 
The possible number of consistent systems of identities originated from the perfect system is
\[
\binom{\frac{p-5}{4}}{\frac{p-5}{8}}.
\]
\end{rem}
\begin{rem}\label{rem2}
The system of identities \eqref{21} can be specified by $2^{\frac{p-5}{2}}$ possibilities by changing the order of the values in the pairs $\{\alpha_{i}, \gamma_{i}\}$. For example, following Table \ref{tab:tab1}, the pair $\{\alpha_2, \gamma_2\}=\{2, 7\}$ can be changed by  $\{\alpha_2, \gamma_2\}=\{7, 2\}$, and so on.
\end{rem}

\subsection{The final part of the proof of Theorem \ref{thm1}}\label{S2.4}

Our further goal is to adapt the statement of Lemma \ref{prop1} to the residues modulo $p$ of the values $\{2!, 3!,\ldots, (p-1)!\}$. The number of residues is $p-2$.  If the aforementioned residues of $\{2!, 3!,\ldots, (p-1)!\}$ all are distinct, then it is known \cite{RS, T} that the missing residue is $r=-\big(\frac{p-1}{2}\big)! \ (\mathrm{mod}\ p)$ that is not congruent to any of $\{2!, 3!,\ldots, (p-1)!\}$. So, we are to complement our set with this additional value. Then the last identity in \eqref{1},
\[
\left(\frac{p-1}{2}\right)!\left(\frac{p-1}{2}\right)!\equiv{-1}(\mathrm{mod}\ p)
\]
is to be replaced by
\[
r\left(\frac{p-1}{2}\right)!\equiv{1}\ (\mathrm{mod}\ p),
\]
and instead of \eqref{1} we have the following system of identities (rewritten here in a more convenient form):
\begin{eqnarray}
2!(p-3)! &\equiv&-1\ (\mathrm{mod}\ p),\nonumber\\
3!(p-4)! &\equiv&+1\ (\mathrm{mod}\ p),\nonumber\\
\cdots&\equiv&\cdots\label{24}\\
(i-1)!(p-i)!&\equiv&(-1)^i\ (\mathrm{mod}\ p),\nonumber\\
\cdots&\equiv&\cdots\nonumber\\
\left(\frac{p-3}{2}\right)!\left(\frac{p+1}{2}\right)!&\equiv&+1\ (\mathrm{mod}\ p),\nonumber
\end{eqnarray}
with the two additional identities:
\begin{eqnarray}
(p-1)!(p-2)!&\equiv&-1 \ (\mathrm{mod}\ p),\label{22}\\
r\left(\frac{p-1}{2}\right)!&\equiv&{+1}\ (\mathrm{mod}\ p)\label{23}.
\end{eqnarray}
That is, we have $\frac{p-1}{2}$ identities in total.

Now we are to consider the two cases: $p=1\ (\mathrm{mod}\ 8)$ and $p=5\ (\mathrm{mod}\ 8)$.

In the case of $p=1\ (\mathrm{mod}\ 8)$, the statement of Theorem \ref{thm1} follows directly from Lemma \ref{prop1}. In this case, the residues modulo $p$ of $\{2!$, $3!$,\ldots, $(p-1)!\}$ cannot be distinct, since then, according to Lemma \ref{prop1},  this contradict to identities \eqref{24}--\eqref{23}.

For the case $p=5\ (\mathrm{mod}\ 8)$ we apply our findings given under the proof of Lemma \ref{prop1}.
With $(p-2)!=1\ (\mathrm{mod}\ p)$ and $(p-1)!=-1\ (\mathrm{mod}\ p)$, identity \eqref{22} is the analogue of the first identity in \eqref{4}. The role of identity \eqref{23} is the tag identity.
Therefore we now study \eqref{24}.

We consider the system
\begin{equation}\label{26}
(i-1)!(p-i)!\equiv(-1)^i \ (\mathrm{mod}\ p), \quad i=3, 4,\ldots, \frac{p-1}{2}.
\end{equation}
Multiplying both sides of \eqref{26} by $(-1)^i$, we obtain
\begin{equation}\label{6}
(-1)^i(i-1)!(p-i)!\equiv 1 \ (\mathrm{mod}\ p), \quad i=3, 4,\ldots, \frac{p-1}{2}.
\end{equation}

As well, keeping in mind that $(p-2)!\equiv1\ (\mathrm{mod}\ p)$, we rewrite \eqref{6} as
\begin{equation}\label{8}
(-1)^i(i-1)!(p-i)![(p-2)!]^{i-1}\equiv 1 \ (\mathrm{mod}\ p), \quad i=3, 4,\ldots, \frac{p-1}{2}.
\end{equation}

We now prove that the system of identities \eqref{8} (or equivalently \eqref{6}) is not perfect.

Indeed, setting
\begin{equation*}
\alpha_i^\prime:=\prod_{j=2}^{i}\alpha_j=i!, \quad \gamma_i^\prime:=\prod_{j=2}^{i}\gamma_j=(-1)^{i+1}(p-i-1)![(p-2)!]^{i},
\end{equation*}
we obtain:
\begin{equation}\label{35}
\alpha_i=i, \quad \gamma_i=-\frac{(p-2)!}{p-i}=\frac{(p-2)!}{i} (\mathrm{mod}\ p).
\end{equation}
We will prove below that if \eqref{8} is the perfect system, that is,
\begin{equation}\label{25}
\alpha_i^\prime\gamma_i^\prime\equiv1\ (\mathrm{mod} \ p)
\end{equation}
is the perfect system, then
\begin{equation}\label{29}
\alpha_i\gamma_i\equiv1\ (\mathrm{mod} \ p), \quad i=2, 3,\ldots, \frac{p-3}{2},
\end{equation}
must be the perfect system too.
Specifically, we have the following lemma.

\begin{lem}\label{prop2} If the system of identities \eqref{25} is perfect, then the following two conditions must be satisfied:

\begin{enumerate}
\item [(i)] $\gamma_i>\frac{p-3}{2}$ for any $i\in\big\{2, 3,\ldots,\frac{p-3}{2}\big\}$.

\item[(ii)] There are no index values $i$ taking the values from $\big\{2, 3,\ldots,\frac{p-3}{2}\big\}$, such that $\gamma_i=p-\alpha_i$.
\end{enumerate}
\end{lem}

\begin{proof} (i) Assume to the contrary that $\gamma_{i_0}=j_0$ for some $i_0$, where $j_0\in\big\{2, 3,\ldots,\frac{p-3}{2}\big\}$. Then
\begin{eqnarray*}
\alpha_{i_0}\alpha_{j_0}=\gamma_{i_0}\gamma_{j_0}=i_0j_0\equiv1 (\mathrm{mod}\ p).
\end{eqnarray*}
If these identities are fulfilled, then the set of all $\alpha_i^\prime, \gamma_i^\prime$ must contain repeated values,
 and hence \eqref{25} must contain repeated identities.
 Then \eqref{25} is not perfect.

\smallskip
(ii) Assume to the contrary that $i_0>2$ is the index value such that $\alpha_{i_0}=\min\{r, p-r\}$. Then $\gamma_{i_0}=\max\{r, p-r\}$, and for all $i\geq i_0$ we have:
\begin{equation}\label{33}
\alpha_{i}^\prime=\min\{r, p-r\}\prod_{\substack{j=2\\ j\neq i_0}}^{i}\alpha_j, \quad \gamma_{i}^\prime=\max\{r, p-r\}\prod_{\substack{j=2\\ j\neq i_0}}^{i}\gamma_j.
\end{equation}
Now we take into account the fact that was earlier used in the proof of Lemma \ref{prop1}: if $\alpha_{i}^\prime\gamma_{i}^\prime\equiv1\ (\mathrm{mod}\ p)$ belongs to the perfect system of identities, then
$(-\alpha_{i}^\prime)(-\gamma_{i}^\prime)\equiv1\ (\mathrm{mod}\ p)$ also does. Therefore together with \eqref{33}, we must also have:

\begin{equation}\label{34}
\alpha_{k(i)}^\prime=\min\{r, p-r\}\prod_{\substack{j=2\\ j\neq i_0}}^{i}\gamma_j, \quad \gamma_{k(i)}^\prime=\max\{r, p-r\}\prod_{\substack{j=2\\ j\neq i_0}}^{i}\alpha_j.
\end{equation}
for some indices $k(i)$. Note that the total number of $\alpha_{i}^\prime$ containing the multiplier $\min\{r, p-r\}$ is even and equal to
\[
\frac{p-1}{2}-i_0.
\]
Hence half of them are associated with $\alpha_{i}^\prime$, and the rest with $\alpha_{k(i)}^\prime$. This means that if the required index value $i_0$ exists, it must only take an even value.

Then for half of the indices defined as $i_0\leq i\leq \frac{p-5+2i_0}{4}$ from \eqref{33} we have $\alpha_{i}^\prime/\alpha_{i-1}^\prime=\alpha_i$. %

Therefore, this subset of indices belongs to the set
\[
\mathcal{I}[i]=\left\{i_0, i_0+1,\ldots, \frac{p-5+2i_0}{4}\right\}.
\]
Then the remaining set of indices $k(i)$ associated with $\alpha_{k(i)}^\prime$ defined in \eqref{34} must belong to the set
\[
\mathcal{I}[k(i)]=\left\{\frac{p-1+2i_0}{4}, \frac{p+3+2i_0}{4},\dots, \frac{p-3}{2}\right\}.
\]

We study the following two cases: (j) the set $\mathcal{I}[i]$ contains at least two elements; (jj) there is only a single element in the set $\mathcal{I}[i]$.
\smallskip

\textit{Case (j)}. Notice that following \eqref{33} and \eqref{34}, for any $i-1$ and $i$ both belonging to $\mathcal{I}[i]$, we have
\begin{equation}\label{36}
\alpha_i=\frac{\alpha_{i}^\prime}{\alpha_{i-1}^\prime}=\frac{\gamma_{k(i)}^\prime}{\gamma_{k(i-1)}^\prime}.
\end{equation}
Assume first for some $i$ and $i-1$, the values
$k(i)$ and $k(i-1)$ are consecutive, i.e., $k(i)=k(i-1)+1$. Then according to part (i), the system of identities \eqref{25} is not perfect.

Assume that there are no $i$ and $i-1$ such that $k(i)$ and $k(i-1)$ are consecutive. Apparently there must be at least one value of $i$ such that $k(i-1)<k(i)$, since otherwise $k(i)$ must follow in descending order, and we arrive at the contradiction, since from \eqref{36} we get
\[
\alpha_i=\frac{\alpha_i^\prime}{\alpha_{i-1}^\prime}=\frac{\gamma_{k(i)}^\prime}{\gamma_{k(i-1)}^\prime}=\frac{\alpha_{k(i-1)}^\prime}{\alpha_{k(i)}^\prime}=\frac{\alpha_{k(i-1)}^\prime}{\alpha_{k(i-1)-1}^\prime}
\]
for all $i-1, i\in \mathcal{I}[i]$.

Under the made assumptions, we use the representations:
\[
\alpha_{l}^\prime=\alpha_2^\prime\prod_{j=3}^{l}\frac{\alpha_j^\prime}{\alpha_{j-1}^\prime}, \quad \gamma_{l}^\prime=\gamma_2^\prime\prod_{j=3}^{l}\frac{\gamma_j^\prime}{\gamma_{j-1}^\prime}, \quad l=3, 4,\ldots, \frac{p-3}{2},
\]
and consequently
\begin{equation}\label{32}
\alpha_{l}^\prime\gamma_{l}^\prime=\alpha_2^\prime\gamma_2^\prime\prod_{j=3}^{l}\left(\frac{\alpha_j^\prime}{\alpha_{j-1}^\prime}\right)\left(\frac{\gamma_j^\prime}{\gamma_{j-1}^\prime}\right), \quad l=3, 4,\ldots, \frac{p-3}{2}
\end{equation}
coming from the telescoping rule. Then for $l\geq k(i)$ we have the following.  Under appropriate change of the order of the terms in \eqref{32}, the fraction
\[
\left(\frac{\alpha_{k(i)}^\prime}{\alpha_{k(i-1)}^\prime}\right)\left(\frac{\gamma_{k(i)}^\prime}{\gamma_{k(i-1)}^\prime}\right)
\]
is one of the multipliers in the product of \eqref{32}, and
\[
\left(\frac{\alpha_{i}^\prime}{\alpha_{i-1}^\prime}\right)\left(\frac{\alpha_{k(i)}^\prime}{\alpha_{k(i-1)}^\prime}\right)\equiv1\ (\mathrm{mod}\ p).
\]
This reduces to the case that have been considered in part (i), and hence the system of identities \eqref{25} is not perfect.

\textit{Case (jj)}.
If the set $\mathcal{I}[i]$ contains a single element, then $i_0$ must be equal to $\frac{p-5}{2}$. Then $\alpha_{i_0}^\prime=\big(\frac{p-5}{2}\big)!$, and hence
$
\frac{p-5}{2}
$
must be equal to $\min\{r, p-r\}$. That is,
\[
\left(\frac{p-5}{2}\right)\left(\frac{p+5}{2}\right)\equiv1\ (\mathrm{mod}\ p).
\]
We show below that such the value $p$ among all of the primes $p=5\ (\mathrm{mod}\ 8)$ is unique and defined for $p=29$ only.

Indeed, for $n\in\mathbb{N}$ consider $f(n)=8n+5$. Then
\begin{equation*}
\frac{f(n)-5}{2}\cdot \frac{f(n)+5}{2}=16n^2+20n\equiv(2n-5)\ (\mathrm{mod}\ 8n+5),
\end{equation*}
and  $2n-5=1$ if and only if $n=3$, that is, the equality is satisfied just for $p=29$. However \eqref{25} is not perfect for $p=29$. The lemma is proved.
\end{proof}

Due to Lemma \ref{prop2}, it is only remained to check whether
\begin{equation}\label{30}
i\cdot\frac{(p-2)!}{i}\equiv1 \ (\mathrm{mod}\ p), \quad i=2, 3,\ldots, \frac{p-3}{2}
\end{equation}
is the perfect system for at least a certain value of $p$.

Let $\delta_i$ denote the smallest positive integer that is congruent to $\frac{(p-2)!}{i}$ modulo $p$. Then
\begin{equation}\label{7}
i\delta_i=K_ip+1\equiv1\ (\mathrm{mod} \ p), \quad i=2, 3,\ldots, \frac{p-3}{2},
\end{equation}
for some integers $K_i$.

Following all this,
the system of identities \eqref{7} can be perfect if all $\delta_i$ are distinct, take the values from
\begin{equation}\label{10}
\left\{\frac{p-1}{2}, \frac{p+1}{2}, \frac{p+3}{2},\ldots, p-2\right\},
\end{equation}
and do not congruent to $\big(\frac{p-1}{2}\big)!\ (\mathrm{mod}\ p)$ and $r\ (\mathrm{mod}\ p)$ that appear in tag identity \eqref{23}.

If \eqref{7} is not the perfect system of identities, then together with \eqref{7} the systems of identities \eqref{6} and \eqref{8} must be not perfect too. Our aim now is to demonstrate that the system of identities \eqref{7} is not perfect.

Note first that neither of $\delta_i$ can be equal to $\frac{p-1}{2}$ or $\frac{p+1}{2}$, since otherwise it must be in contradiction with tag identity \eqref{23}. Indeed, if one of the $\delta_i$ is equal to $\frac{p-1}{2}$ or $\frac{p+1}{2}$, then we must assume that
\begin{equation}\label{9}
\left(\frac{p-1}{2}\right)!\neq \pm\frac{p-1}{2}\ (\mathrm{mod}\ p).
\end{equation}
The last is impossible, since then either $\big(\frac{p-1}{2}\big)!\ (\mathrm{mod}\ p)$ or $r\ (\mathrm{mod}\ p)$ must be congruent to one of the numbers $i=2, 3,\ldots,\frac{p-3}{2}$, that is, coincide with one of $\alpha_i=i$.

So, the remaining case is
\begin{equation}\label{11}
\left(\frac{p-1}{2}\right)!\equiv \pm\frac{p-1}{2}\ (\mathrm{mod}\ p).
\end{equation}
This case is impossible too.

Indeed, if $\big(\frac{p-1}{2}\big)!\equiv\frac{p-1}{2} \ (\mathrm{mod}\ p)$, then we need to have $\big(\frac{p-3}{2}\big)!\equiv1\ (\mathrm{mod}\ p)$. But this contradicts to the fact that all residues of $\{2!, 3!,\ldots, (p-1)!\}$ are distinct, since $(p-2)!\equiv1\ (\mathrm{mod}\ p)$ as well. Similarly, the identity $\big(\frac{p-1}{2}\big)!\equiv-\frac{p-1}{2} \ (\mathrm{mod}\ p)$ cannot be satisfied either, since then we need to have $\big(\frac{p-3}{2}\big)!\equiv-1\ (\mathrm{mod}\ p)$, and we arrived at the contradiction again, since $(p-1)!\equiv-1\ (\mathrm{mod}\ p)$.

Thus, neither  \eqref{9} nor \eqref{11} leads to the perfect system. The system of identities \eqref{7} is not perfect, and together with it the systems of identities \eqref{6} and \eqref{8} are not perfect either. The theorem is proved.

\subsection*{Declaration of interest statements}

\subsubsection*{Disclosure of interest} No conflict of interests was reported by the author.

\subsubsection*{Declaration of funding} No funding for this research was received.

\subsubsection*{Data availability statement} Data sharing is not applicable to this article as no new data were created or analyzed in this study.

\end{document}